\documentclass[11pt]{article}
\usepackage{amsmath,amsthm,amscd,latexsym}
\usepackage{amsfonts,pdfsync,color,graphicx}
\usepackage[psamsfonts]{amssymb}
\usepackage{epsfig}
\usepackage{color}
\usepackage[noadjust]{cite}
\usepackage{accents}
\usepackage{listings}
\usepackage[top=1.5in,bottom=1.5in,left=1.3in,right=1.0in]{geometry}
\usepackage{multirow}

\newtheorem{theorem}{Theorem}

\newtheorem{definition}[theorem]{Definition}

\newtheorem{lemma}[theorem]{Lemma}

\newtheorem{remark}[theorem]{Remark}

\DeclareMathOperator{\dif}{d}

\begin{document}

\title{On multi-point resonant problems on the half-line}
\date{}
\author{Luc{\' i}a L\'opez-Somoza$^1$ and Feliz Minh\'os$^{2,3}$ \\
$^1$ Departamento de Estat\'istica, An\'alise Matem\'atica e Optimizaci\'on \\ Instituto de Ma\-te\-m\'a\-ti\-cas, Facultade de Matem\'aticas, \\
Universidade de Santiago de Com\-pos\-te\-la, 15782 Santiago de Compostela,\\
Galicia, Spain.\\
lucia.lopez.somoza@usc.es\\
$^2$ Departamento de Matem\'atica, Escola de Ci\^encias e Tecnologia,\\
$^3$ Centro de Investiga\c{c}\~ao em Matem\'atica e Aplica\c{c}\~oes (CIMA),%
\\
Instituto de Investiga\c{c}\~ao e Forma\c{c}\~ao Avan\c{c}ada,\\
Universidade de \'Evora, \'Evora, Portugal.\\
fminhos@uevora.pt}
\maketitle

\begin{abstract}
In this work we obtain sufficient conditions for the existence of bounded
solutions of a resonant multi-point second-order boundary value problem,
with a fully differential equation.

The noninvertibility of the linear part is overcome by a new perturbation
technique, which allows to obtain an existence result and a localization
theorem. Our hypotheses are clearly much less restrictive than the ones
existent in the literature and, moreover, they can be applied to higher
order, resonant or non-resonant, boundary value problems defined on the
half-line or even on the real line.
\end{abstract}


\section{Introduction}

In this paper, we will prove the existence of bounded solutions for the
multi-point boundary value problem 
\begin{equation}
\left\{ 
\begin{split}
& u^{\prime \prime }(t)=f(t,u(t),u^{\prime }(t)),\quad t\in \lbrack 0,\infty
), \\
& u(0)=0,\ u^{\prime }(+\infty )=\sum_{i=1}^{m-1}\alpha _{i}\,u^{\prime
}(\xi _{i}),
\end{split}%
\right.   \label{e-multipoint}
\end{equation}%
where $\alpha _{i}>0$ and $0=\xi _{1}<\cdots <\xi _{m-1}<+\infty $. We
assume that the coefficients $\alpha _{i}$ satisfy the resonant condition 
\begin{equation}
\sum_{i=1}^{m-1}\alpha _{i}=1.  \label{e-resonant}
\end{equation}

A boundary value problem is said to be resonant when the correspondent
homogeneous problem has nontrivial solutions. In fact, under condition %
\eqref{e-resonant}, the homogeneous boundary value problem related to %
\eqref{e-multipoint}, 
\begin{equation}
\left\{ 
\begin{split}
& u^{\prime \prime }(t)=0,\quad t\in \lbrack 0,\infty ), \\
& u(0)=0,\ u^{\prime }(+\infty )=\sum_{i=1}^{m-1}\alpha _{i}\,u^{\prime
}(\xi _{i}),
\end{split}%
\right.   \label{e-multipoint-homog}
\end{equation}%
has a nontrivial solution.

These resonant problems have been studied for many years under a huge
variety of arguments: degree theory has been used in, for instance, \cite%
{Bebernes, Feng, Kauf, Rach}, Lyapunov--Schmidt arguments, \cite{Ma}, a
Leggett--Williams theorem \cite{Franco, O'Regan}, fixed point and fixed
point index theories, \cite{Bai, Han, Infante, Yang}, monotone method
together with upper and lower solutions technique, \cite{Mosa}, among others.

Boundary value problems on unbounded intervals arise in many models of
applied mathematics, such as in combustion theory, in plasma physics, to
model the unsteady flow of a gas through semi-infinite porous media, to
study the electrical potential of an isolated neutral atom, ... For more
details, techniques and applications in this field we refer, for example, to
\cite{Jiang, Kim, Kosma, Lian, Xu}, and the monograph \cite{AgarRegInfinit}.

In a theoretical point of view, resonance problems can be formulated as an
equation $Lx=Nx$, where $L$ is a noninvertible operator. Therefore, in
particular, the resonant condition \eqref{e-resonant} implies that the
Green's function related to problem \eqref{e-multipoint-homog} does not
exist. This issue is overcome applying several techniques. For instance, in 
\cite{JiangYang2016} the authors studied the problem 
\begin{equation*}
\left\{ 
\begin{split}
& u^{\prime \prime }(t)+f(t,u(t))=0,\quad t\in \lbrack 0,\infty ), \\
& u(0)=0,\ u^{\prime }(+\infty )=\sum_{i=1}^{m-1}\alpha _{i}\,u^{\prime
}(\xi _{i}),
\end{split}%
\right. 
\end{equation*}%
also under condition \eqref{e-resonant} and, to deal with the resonance
problem they defined some suitable operators and were able to find a
solution in the space 
\begin{equation*}
E=\left\{ u\in {\mathcal{C}}[0,\infty ),\ u(0)=0,\ \sup_{t\in \lbrack
0,\infty )}\frac{|u(t)|}{1+t}<+\infty \right\} ,
\end{equation*}%
so clearly that solution could be unbounded.

Our arguments apply a different technique to find bounded solutions for
problem \eqref{e-multipoint}. Moreover, we note that, on the contrary to 
\cite{JiangYang2016}, we allow the nonlinearity $f$ to depend on the first
derivative of $u$.

In \cite{DjebaliGuedda2017}, a similar third order boundary value problem is
considered, namely 
\begin{equation*}
\left\{%
\begin{split}
& u^{\prime \prime \prime }(t)=f(t,u(t),u^{\prime }(t),u^{\prime \prime
}(t))=0, \quad t\in[0,\infty), \\
& u(0)=u^{\prime }(0)=0, \ u^{\prime \prime }(+\infty)=\sum_{i=1}^{m-2}
\alpha_i \,u^{\prime \prime }(\xi_i),
\end{split}%
\right.
\end{equation*}
coupled with the resonant condition 
\begin{equation*}
\sum_{i=1}^{m-2} \alpha_i=1.
\end{equation*}
The techniques used in \cite{DjebaliGuedda2017} are basically the same than
in \cite{JiangYang2016} and, again, the authors are able to find a solution
which could be unbounded. On the other hand, they allow the nonlinearity $f$
to depend on all the derivatives up to the highest possible order but, to do
that, they asked for the following quite restrictive condition on the
nonlinearity:

\begin{itemize}
\item[($H_0$)] $f:[0,+\infty)\times {\mathbb{R}}^3\rightarrow {\mathbb{R}}$
is $s^2$-Carath\'eodory, that is,

\begin{itemize}
\item[(i)] $f(\cdot,u,v,w)$ is measurable for each $(u,v,w)$ fixed.

\item[(ii)] $f(t,\cdot,\cdot,\cdot)$ is continuous for a.\,e. $%
t\in[0,\infty) $.

\item[(iii)] For each $r>0$ there exists $\psi _{r}\in L^{1}[0,\infty )$
with $t\,\psi _{r},\,t^{2}\,\psi _{r}\in L^{1}[0,\infty )$ such that 
\begin{equation*}
\left\vert f(t,u,v,w)\right\vert \leq \psi _{r}(t),\quad \forall
\,(u,v,w)\in (-r,r)\times (-r,r)\times (-r,r),\ \text{a.\thinspace e. }t\in
\lbrack 0,\infty ).
\end{equation*}
\end{itemize}
\end{itemize}

Here, we must point out that, although in this paper we work with the second
order problem, the same techniques could be applied to the third order
problem. In this sense, we allow the nonlinearity $f$ to depend on all the
derivatives up to the highest possible order but using either hypothesis $%
(H_{1})$ or $(H_{2})$ instead of $(H_{0})$. This way, our hypotheses
are clearly much less restrictive than $(H_{0})$ so our method improves the
results in \cite{DjebaliGuedda2017}.

We would also like to mention that our technique of modifying the problem,
in order to obtain another one with a related Green's function in $%
L^{1}[0,\infty )\cap L^{\infty }[0,\infty )$, is also applicable to problems
without resonance. Thus, if we used this idea in problems like 
\begin{equation*}\left\{ 
\begin{split}
& u^{(4)}(t)+k\,u(t)=f(t,u(t),u^{\prime }(t),u^{\prime \prime }(t),u^{\prime
\prime \prime }(t)),\ t\in {\mathbb{R}}, \\[2pt]
& u(\pm \infty )=0,\ u^{\prime }(\pm \infty )=0,
\end{split}\right.\end{equation*}
considered in \cite{MinCar-4ord}, we could extend the results in that
reference to nonlinearities satisfying $(H_{2})$ instead of $(H_{1})$. The
same could be said about \cite{MinCar}.

The paper is divided into several sections: In Section 2, we construct an
auxiliary differential problem whose solutions are the same than those of
problem \eqref{e-multipoint}. In Section 3, this auxiliary problem is
transformed into an integral one, for which some bounded solutions are
found. These solutions are showed to be solutions of the original problem.
Finally, Section 4 includes an example which can not be solved with the
results in \cite{JiangYang2016}.

\section{Preliminaries}

We will construct now a modified problem, which will be shown equivalent to %
\eqref{e-multipoint}, for which it is possible to construct the related
Green's function.

Indeed, consider the modified problem 
\begin{equation}  \label{e-multipoint-mod}
\left\{%
\begin{split}
& u^{\prime \prime }(t)+k\,u^{\prime }(t)+M\,u(t)=0, \quad t\in[0,\infty), \\
& u(0)=0, \ u^{\prime }(+\infty)=\sum_{i=1}^{m-1} \alpha_i \,u^{\prime
}(\xi_i),
\end{split}%
\right.
\end{equation}
where $k$ and $M$ are positive numbers such that $k^2-4\,M<0$ and 
\begin{equation*}
{\displaystyle\sum_{i=1}^{m-1}} \alpha_i\, e^{-\frac{k\,\xi_i}{2}} \left(-%
\frac{k}{2} \sin{(\gamma \,\xi_i)} + \gamma\,\cos{(\gamma\,\xi_i)}
\right)\neq 0,
\end{equation*}
with $\gamma=\sqrt{4\,M-k^2}$.

If we denote by
\begin{equation*}
h_l(s)=\frac{{\displaystyle\sum_{i=l}^{m-1}} \alpha_i\, e^{-\frac{k\,\xi_i}{2%
}} \left(-\frac{k}{2} \sin{(\gamma\,(s-\xi_i))} + \gamma\,\cos{%
(\gamma\,(s-\xi_i))} \right)} {{\displaystyle\sum_{i=1}^{m-1}} \alpha_i\,
e^{-\frac{k\,\xi_i}{2}} \left(-\frac{k}{2} \sin{(\gamma \,\xi_i)} +
\gamma\,\cos{(\gamma\,\xi_i)} \right)},
\end{equation*}
then the Green's function related to problem \eqref{e-multipoint-mod} is
given by the following expression: 
\begin{equation*}
G(t,s)=\frac{1}{\gamma}\,e^{-\frac{k\,(t+s)}{2}}\left\{%
\begin{array}{ll}
-\sin{(\gamma\,t)} \, h_l(s), & 0\le t \le s,\ \ \xi_{l-1} \le s <\xi_l, \\%
[.3cm] 
0, & 0\le t \le s,\ \ \xi_{m-1} \le s, \\[.3cm] 
-\sin{(\gamma\,t)} \, h_l(s)+\sin{(\gamma\,(s-t))}, & 0\le s < t, \ \
\xi_{l-1} \le s <\xi_l, \\[.3cm] 
\sin{(\gamma\,(s-t))}, & 0\le s < t,\ \ \xi_{m-1} \le s.%
\end{array}%
\right.
\end{equation*}

The first derivative of the Green's function is given by 
\begin{equation*}
\frac{\partial\,G}{\partial\, t}(t,s) =\frac{1}{\gamma}\,e^{-\frac{k\,(t+s)}{%
2}}\left\{%
\begin{array}{ll}
\left(\frac{k}{2}\,\sin{(\gamma\,t)} -\gamma\,\cos{(\gamma\,t)}\right)
\,h_l(s), & 0\le t \le s,\ \ \xi_{l-1} \le s <\xi_l, \\[.3cm] 
0, & 0\le t \le s,\ \ \xi_{m-1} \le s, \\[.3cm] 
\begin{array}{l}
\left(\frac{k}{2}\,\sin{(\gamma\,t)} -\gamma\,\cos{(\gamma\,t)}\right) \,
h_l(s) \\[1pt] 
-\frac{k}{2}\sin{(\gamma\,(s-t))}-\cos{(\gamma\,(s-t))},%
\end{array}
& 0\le s < t,\ \ \xi_{l-1} \le s <\xi_l, \\[.6cm] 
-\frac{k}{2}\sin{(\gamma\,(s-t))}-\cos{(\gamma\,(s-t))}, & 0\le s < t,\ \
\xi_{m-1} \le s.%
\end{array}%
\right.
\end{equation*}

\begin{remark}
\label{r-bound-G-Gprima} It is easy to see that there exist two positive
constants, $C_1$ and $C_2$, such that 
\begin{equation*}
\left|G(t,s)\right|\le C_1\, e^{-\frac{k(t+s)}{2}}
\end{equation*}
and 
\begin{equation*}
\left|\frac{\partial\,G}{\partial\, t}(t,s)\right|\le C_2\, e^{-\frac{k(t+s)%
}{2}},
\end{equation*}
for all $(t,s)\in [0,\infty)\times [0,\infty)$.

As a consequence, it is clear that both $G(t,\cdot)$ and $\frac{\partial\,G}{%
\partial\,t}(t,\cdot)$ belong to $L^1[0,\infty) \cap L^\infty[0,\infty)$ for
all $t\in[0,\infty)$.
\end{remark}

On the other hand, to deal with the lack of compactness of the set $X$ that
we will consider in the following section, we will use the following result:

\begin{theorem}[\protect\cite{Corduneanu}]
\label{thm_comp_criteria_corduneanu} Let $E$ be a Banach space and ${%
\mathcal{C}}(\mathbb{R},E)$ the space of all bounded continuous functions $%
x\colon {\mathbb{R}}\rightarrow E$. For a set $D\subset {\mathcal{C}}({%
\mathbb{R}},E)$ to be relatively compact, it is necessary and sufficient
that:

\begin{enumerate}
\item $D$ is uniformly bounded;

\item functions from $D$ are equicontinuous on every compact subinterval of $%
[0,\infty)$;

\item functions from $D$ are equiconvergent at $+\infty$, that is, given $%
\varepsilon>0$, there exists $T>0$ such that for all $t\ge T$, we have that 
\begin{equation*}
\left\|x(t)-\lim\limits_{t\to \infty} x(t) \right\|<\varepsilon.
\end{equation*}
\end{enumerate}
\end{theorem}

To prove the existence of solutions we will consider two different results.
First of all we will use the very well-known Schauder's fixed point theorem:

\begin{theorem}[\protect\cite{Zeidler1986}]
\label{t-Schauder} Let $Y$ be a nonempty, closed, bounded and convex subset
of a Banach space $X$, and suppose that $P\colon Y\rightarrow Y$ is a
compact operator. Then $P$ has at least one fixed point in $Y$.
\end{theorem}

On the other hand, we will also give a result to prove the existence of
solutions based on the lower and upper solutions technique. To do that we
need to introduce the following definition:

\begin{definition}
A function $\alpha\in X$ is said to be a lower solution of problem %
\eqref{e-multipoint} if 
\begin{equation*}
\left\{%
\begin{split}
& \alpha^{\prime \prime }(t)\ge f(t,\alpha(t),\alpha^{\prime }(t)), \quad
t\in [0,\infty), \\
& \alpha(0)\le 0, \ \alpha^{\prime }(+\infty)\ge
\sum_{i=1}^{m-1}\alpha_i\,\alpha^{\prime }(\xi_i).
\end{split}%
\right.
\end{equation*}
A function $\beta\in X$ is said to be an upper solution of %
\eqref{e-multipoint} if the reversed inequalities hold.
\end{definition}

\section{Main results}

Let us consider 
\begin{equation*}
X=\left\{u\in{\mathcal{C}}^1[0,\infty)\,: \
\sup_{t\in[0,\infty)}|u(t)|<\infty, \ \sup_{t\in[0,\infty)}|u^{\prime
}(t)|<\infty \right\}
\end{equation*}
equipped with the norm 
\begin{equation*}
\|u\|=\max\left\{\|u\|_\infty, \ \|u^{\prime }\|_\infty \right\},
\end{equation*}
where $\|v\|_\infty=\sup_{t\in[0,\infty)} |v(t)|$. It is easy to prove that $%
(X,\, \|\cdot\|)$ is a Banach space.

Consider the following integral operator $T\colon X\rightarrow X$ defined by 
\begin{equation}
Tu(t)=\int_{0}^{\infty }G(t,s)\,\left( f(s,u(s),u^{\prime }(s))+k\,u^{\prime
}(s)+M\,u(s)\right) \,\mbox{\,$\rm d$}s.  \label{e-int-T}
\end{equation}%
It is clear that solutions of problem \eqref{e-multipoint} are fixed points
of operator $T$.

Moreover, we will assume that at least one of the two following conditions
holds:

\begin{itemize}
\item[$(H_1)$] The nonlinearity $f:[0,\infty) \times {\mathbb{R}}^2
\rightarrow {\mathbb{R}}$ satisfies $L^1$-Carath\'eodory condition, that is,

\begin{itemize}
\item[(i)] $f(\cdot,u,v)$ is measurable for each $(u,v)$ fixed.

\item[(ii)] $f(t,\cdot,\cdot)$ is continuous for a.\,e. $t\in[0,\infty)$.

\item[(iii)] For each $r>0$ there exists $\varphi _{r}\in L^{1}[0,\infty )$
such that 
\begin{equation*}
\left\vert f(t,u,v)\right\vert \leq \varphi _{r}(t),\quad \forall \,(u,v)\in
(-r,r)\times (-r,r),\ \text{a.\thinspace e. }t\in \lbrack 0,\infty ).
\end{equation*}
\end{itemize}

\item[$(H_2)$] The nonlinearity $f:[0,\infty) \times {\mathbb{R}}^2
\rightarrow {\mathbb{R}}$ satisfies $L^\infty$-Carath\'eodory condition,
that is,

\begin{itemize}
\item[(i)] $f(\cdot,u,v)$ is measurable for each $(u,v)$ fixed.

\item[(ii)] $f(t,\cdot,\cdot)$ is continuous for a.\,e. $t\in[0,\infty)$.

\item[(iii)] For each $r>0$ there exists $\phi _{r}\in L^{\infty }[0,\infty
) $ such that 
\begin{equation*}
\left\vert f(t,u,v)\right\vert \leq \phi _{r}(t),\quad \forall \,(u,v)\in
(-r,r)\times (-r,r),\ \text{a.\thinspace e. }t\in \lbrack 0,\infty ).
\end{equation*}
\end{itemize}
\end{itemize}

Under one of these conditions we will be able to prove the following result.

\begin{lemma}
\label{l-T-comp-cont-multipoint} Assume that either $(H_1)$ or $(H_2)$
holds. Then operator $T$ defined in \eqref{e-int-T} is completely continuous.
\end{lemma}

\begin{proof}
The proof will be divided into several steps.

\vspace*{3pt} \noindent \underline{Step 1}: $T$ is well-defined in $X$.

\vspace*{1pt} Given an arbitrary $u\in X$, we will prove that $Tu\in X$.

First, we will make the proof in case hypothesis $(H_{1})$ holds. If $u\in X$%
, then there exists some $r>0$ such that $\Vert u\Vert <r$. Therefore, it
holds that 
\begin{equation}
\begin{split}
|Tu(t)|=& \left\vert \int_{0}^{\infty }G(t,s)\,\left( f(s,u(s),u^{\prime
}(s))+k\,u^{\prime }(s)+M\,u(s)\right) \,\mbox{\,$\rm d$}s\right\vert \\
\leq & \,\int_{0}^{\infty }\left\vert G(t,s)\right\vert \,\left( \left\vert
f(s,u(s),u^{\prime }(s))\right\vert +k\,\left\vert u^{\prime }(s)\right\vert
+M\,\left\vert u(s)\right\vert \right) \,\mbox{\,$\rm d$}s \\
\leq & \,\int_{0}^{\infty }\left\vert G(t,s)\right\vert \,\left( \varphi
_{r}(s)+(k+M)\,r\right) \,\mbox{\,$\rm d$}s \\
\leq & \,\int_{0}^{\infty }C_{1}\,e^{-\frac{k(t+s)}{2}}\,\left( \varphi
_{r}(s)+(k+M)\,r\right) \,\mbox{\,$\rm d$}s \\
=& \,C_{1}\,e^{-\frac{kt}{2}}\left( \int_{0}^{\infty }e^{-\frac{ks}{2}%
}\,\varphi _{r}(s)\,\mbox{\,$\rm d$}s+\frac{2}{k}\,(k+M)\,r\right) \\
=& \,C_{1}\,e^{-\frac{kt}{2}}\left( \int_{0}^{\infty }e^{-\frac{ks}{2}%
}\,\varphi _{r}(s)\,\mbox{\,$\rm d$}s+\left( 2+\frac{2\,M}{k}\right)
\,r\right)
\end{split}
\label{ec-bound-Tu}
\end{equation}%
and, analogously, 
\begin{equation}
\begin{split}
\left\vert (Tu)^{\prime }(t)\right\vert =& \left\vert \int_{0}^{\infty }%
\frac{\partial \,G}{\partial \,t}(t,s)\,\left( f(s,u(s),u^{\prime
}(s))+k\,u^{\prime }(s)+M\,u(s)\right) \,\mbox{\,$\rm d$}s\right\vert \\
\leq & \,\int_{0}^{\infty }\left\vert \frac{\partial \,G}{\partial \,t}%
(t,s)\right\vert \,\left( \left\vert f(s,u(s),u^{\prime }(s))\right\vert
+k\,\left\vert u^{\prime }(s)\right\vert +M\,\left\vert u(s)\right\vert
\right) \,\mbox{\,$\rm d$}s \\
\leq & \,\int_{0}^{\infty }\left\vert \frac{\partial \,G}{\partial \,t}%
(t,s)\right\vert \,\left( \varphi _{r}(s)+(k+M)\,r\right) \,\mbox{\,$\rm d$}s
\\
\leq & \,\int_{0}^{\infty }C_{2}\,e^{-\frac{k(t+s)}{2}}\,\left( \varphi
_{r}(s)+(k+M)\,r\right) \,\mbox{\,$\rm d$}s \\
=& \,C_{2}\,e^{-\frac{kt}{2}}\left( \int_{0}^{\infty }e^{-\frac{ks}{2}%
}\,\varphi _{r}(s)\,\mbox{\,$\rm d$}s+\left( 2+\frac{2\,M}{k}\right)
\,r\right) .
\end{split}
\label{ec-bound-Tuprima}
\end{equation}%
Now, since $\varphi _{r}\in L^{1}[0,\infty )$ and $e^{-\frac{ks}{2}}\in
L^{\infty }[0,\infty )$, it holds that $\varphi _{r}(s)\,e^{-\frac{ks}{2}%
}\in L^{1}[0,\infty )$. Thus, it is clear that 
\begin{equation*}
\sup_{t\in \lbrack 0,\infty )}|Tu(t)|<\infty \quad \text{and}\quad
\sup_{t\in \lbrack 0,\infty )}\left\vert (Tu)^{\prime }(t)\right\vert
<\infty ,
\end{equation*}%
that is, $Tu\in X$.

On the other hand, if $(H_{2})$ holds instead of $(H_{1})$, following
similar steps to the previous case, we obtain the following upper bounds: 
\begin{equation*}
|Tu(t)|\leq \,C_{1}\,e^{-\frac{kt}{2}}\left( \int_{0}^{\infty }e^{-\frac{ks}{%
2}}\,\phi _{r}(s)\,\mbox{\,$\rm d$}s+\left( 2+\frac{2\,M}{k}\right)
\,r\right)
\end{equation*}%
and 
\begin{equation*}
|(Tu)^{\prime }(t)|\leq \,C_{2}\,e^{-\frac{kt}{2}}\left( \int_{0}^{\infty
}e^{-\frac{ks}{2}}\,\phi _{r}(s)\,\mbox{\,$\rm d$}s+\left( 2+\frac{2\,M}{k}%
\right) \,r\right) .
\end{equation*}

In this case $\phi _{r}\in L^{\infty }[0,\infty )$ and, since $e^{-\frac{ks}{2%
}}\in L^{1}[0,\infty )$, we obtain that $\phi_{r}(s)\,e^{-\frac{ks}{2}}\in
L^{1}[0,\infty )$. Therefore we conclude again that $Tu\in X$.

\vspace*{3pt} \noindent \underline{Step 2}: $T$ is a continuous operator.

\vspace*{1pt} We will detail the proof for the case in which $(H_{1})$
holds. For $(H_{2})$ the proof will be analogous, with the obvious changes, as it occurred in Step 1.

Consider the sequence $\{u_{n}\}_{n\in {\mathbb{N}}}$ and assume that it
converges to $u$ in $X$, that is, 
\begin{equation*}
\lim\limits_{n\rightarrow \infty }u_{n}(t)=u(t)\quad \text{and}\quad
\lim\limits_{n\rightarrow \infty }u_{n}^{\prime }(t)=u^{\prime }(t)\quad 
\text{for all }t\in \lbrack 0,\infty ).
\end{equation*}%
Then, since $f(t,\cdot ,\cdot )$ is continuous for a.\thinspace e. $t\in
\lbrack 0,\infty )$, it is deduced that 
\begin{equation*}
\lim\limits_{n\rightarrow \infty }f(s,u_{n}(s),u_{n}^{\prime
}(s))=f(s,u(s),u^{\prime }(s)) \quad \text{for a.\,e. } s\in[0,\infty).
\end{equation*}%
Let's see that $\{Tu_{n}\}_{n\in {\mathbb{N}}}$ converges to $Tu$.

Since $\{u_{n}\}_{n\in {\mathbb{N}}}$ is convergent in $X$, then there
exists some $r>0$ such that $\Vert u_{n}\Vert <r$ for all $n\in {\mathbb{N}}$. Now, if $(H_{1})$ holds, 
\begin{equation*}\begin{split}
\left\vert Tu_{n}(t)-Tu(t)\right\vert  \leq& \, \int_{0}^{\infty }\left\vert
G(t,s)\right\vert \, \left\vert f(s,u_{n}(s),u_{n}^{\prime
}(s))-f(s,u(s),u'(s))\right\vert \dif s\\[2pt]
& \, + \int_{0}^{\infty } \left\vert
G(t,s)\right\vert \, \left(k\,|u_{n}^{\prime }(s)-u'(s)|+M\,|u_{n}(s)-u(s)|\right) 
\dif s \\[2pt]
\leq& \, C_1 \int_{0}^{\infty } \negthinspace e^{-\frac{k(t+s)}{2}} \left\vert f(s,u_{n}(s),u_{n}^{\prime
}(s))-f(s,u(s),u'(s))\right\vert \dif s\\[2pt]
& \, +  C_1 \int_{0}^{\infty } \negthinspace e^{-\frac{k(t+s)}{2}} \left(k\,|u_{n}^{\prime }(s)-u'(s)|+M\,|u_{n}(s)-u(s)|\right) 
\dif s  \\[2pt]
\le & \, C_1 \int_{0}^{\infty } e^{-\frac{ks}{2}} \left(2 \,\varphi_r(s)+2\,(k+M)\,r \right) \dif s <\infty.
\end{split}\end{equation*}

Then, we deduce from Lebesgue's Dominated Convergence Theorem that
\begin{equation*}\begin{split}
\lim\limits_{n\to \infty} \left\|Tu_n - Tu\right\|_\infty \le & \,\lim\limits_{n\to \infty} C_1 \int_{0}^{\infty }  e^{-\frac{ks}{2}} \left\vert f(s,u_{n}(s),u_{n}^{\prime
}(s))-f(s,u(s),u'(s))\right\vert \dif s\\[2pt]
& \, + \lim\limits_{n\to \infty} C_1  \int_{0}^{\infty } e^{-\frac{ks}{2}} \left(k\,|u_{n}^{\prime }(s)-u'(s)|+M\,|u_{n}(s)-u(s)|\right) 
\dif s \\[2pt]
= & \, C_1 \int_{0}^{\infty } \lim\limits_{n\to \infty} e^{-\frac{ks}{2}} \left\vert f(s,u_{n}(s),u_{n}^{\prime}(s))-f(s,u(s),u'(s))\right\vert \dif s\\[2pt]
& \, + C_1 \int_{0}^{\infty } \lim\limits_{n\to \infty} e^{-\frac{ks}{2}} \left(k\,|u_{n}^{\prime }(s)-u'(s)|+M\,|u_{n}(s)-u(s)|\right) \dif s =0.
\end{split}\end{equation*}

Analogously, we get that
\begin{equation*}\begin{split}
\lim\limits_{n\to \infty} \left\|(Tu_n)' - (Tu)'\right\|_\infty \le & \, C_2 \int_{0}^{\infty } \lim\limits_{n\to \infty} e^{-\frac{ks}{2}} \left\vert f(s,u_{n}(s),u_{n}^{\prime}(s))-f(s,u(s),u'(s))\right\vert \dif s\\[2pt]
& \, + C_2 \int_{0}^{\infty } \lim\limits_{n\to \infty} e^{-\frac{ks}{2}} \left(k\,|u_{n}^{\prime }(s)-u'(s)|+M\,|u_{n}(s)-u(s)|\right) \dif s =0.
\end{split}\end{equation*}

Thus, $\left\{ Tu_{n}\right\} _{n\in {\mathbb{N}}}$ converges to $Tu$ in $X$.

\vspace*{3pt} \noindent \underline{Step 3}: $T$ is compact.

\vspace*{1pt} Again, we will make the proof only for the case in which $%
(H_{1})$ holds, being the other one analogous.

Let $B$ be a bounded subset of $X$, that is, there exists some $r>0$ such
that $\Vert u\Vert <r$, for all $u\in B$. Let us see that $T(B)$ is
relatively compact in $X$.

\begin{description}
\item[(i)] $T(B)$ is uniformly bounded:\newline
If $u\in B$, then, for $t\in \lbrack 0,\infty )$, 
\begin{equation*}\begin{split}
|Tu(t)|&\leq C_{1}\,e^{-\frac{kt}{2}}\,\left( \int_{0}^{\infty }e^{-\frac{ks}{2}}\, \varphi _{r}(s)\, \mbox{\,$\rm d$}s+\left( 2+\frac{2\,M}{k}\right)
\,r\right)\\[2pt]
&\le  C_{1}\,\left( \int_{0}^{\infty }e^{-\frac{ks}{2}}\,\varphi _{r}(s)\,\mbox{\,$\rm d$}s+\left( 2+\frac{2\,M}{k}\right)
\,r\right)=:M_{1}>0
\end{split}\end{equation*}
and 
\begin{equation*}\begin{split}
|(Tu)^{\prime }(t)|&\leq C_{2}\,e^{-\frac{kt}{2}}\,\left( \int_{0}^{\infty
}e^{-\frac{ks}{2}}\,\varphi _{r}(s)\,\mbox{\,$\rm d$}s+\left( 2+\frac{2\,M}{k}\right) \,r\right) \\[2pt] 
&\leq C_{2}\,\left( \int_{0}^{\infty}e^{-\frac{ks}{2}}\,\varphi _{r}(s)\, \mbox{\,$\rm d$}s+\left( 2+\frac{2\,M}{k}\right) \,r\right) =:M_{2}>0.
\end{split}\end{equation*}%
Thus, 
\begin{equation*}
\Vert Tu\Vert \leq \max \{M_{1},\,M_{2}\},
\end{equation*}%
for all $u\in B$, that is, $T(B)$ is uniformly bounded.

\item[(ii)] $T(B)$ is equicontinuous:\newline
Let $t_{1},\,t_{2}\in \lbrack 0,L]$ for some $L>0$ and assume that $%
t_{1}>t_{2}$. Then, 
\begin{equation}
\begin{split}
\left\vert Tu(t_{1})-Tu(t_{2})\right\vert \leq & \int_{0}^{\infty
}\left\vert G(t_{1},s)-G(t_{2},s)\right\vert \,\left( \left\vert
f(s,u(s),u^{\prime }(s))\right\vert +k\,\left\vert u^{\prime }(s)\right\vert
+M\,\left\vert u(s)\right\vert \right) \mbox{\,$\rm d$}s \\
\leq & \,\int_{0}^{\infty }\left\vert G(t_{1},s)-G(t_{2},s)\right\vert
\,\left( \varphi _{r}(s)+(k+M)\,r\right) \mbox{\,$\rm d$}s \\
=& \,\int_{0}^{t_{2}}\left\vert G(t_{1},s)-G(t_{2},s)\right\vert \,\left(
\varphi _{r}(s)+(k+M)\,r\right) \mbox{\,$\rm d$}s \\
& +\int_{t_{2}}^{t_{1}}\left\vert G(t_{1},s)-G(t_{2},s)\right\vert \,\left(
\varphi _{r}(s)+(k+M)\,r\right) \mbox{\,$\rm d$}s \\
& +\int_{t_{1}}^{\infty }\left\vert G(t_{1},s)-G(t_{2},s)\right\vert
\,\left( \varphi _{r}(s)+(k+M)\,r\right) \mbox{\,$\rm d$}s.
\end{split}
\label{e:equiconv_1}
\end{equation}%
We will find some suitable upper bounds for the difference $%
|G(t_{1},s)-G(t_{2},s)|$. For $0\leq t_{2}<t_{1}\leq s$, we have two
possibilities:

\begin{itemize}
\item If $\xi _{l-1}\leq s<\xi _{l}$ for some $2\leq l\leq m-1$, then 
\begin{equation*}
|G(t_{1},s)-G(t_{2},s)|=\frac{1}{\gamma }\,|h_{l}(s)|\,e^{-\frac{k\,s}{2}%
}\,\left\vert -e^{-\frac{k\,t_{1}}{2}}\,\sin (\gamma \,t_{1})+e^{-\frac{%
k\,t_{2}}{2}}\,\sin (\gamma \,t_{2})\right\vert .
\end{equation*}

\item If $s\geq \xi _{m-1}$, then 
\begin{equation*}
|G(t_{1},s)-G(t_{2},s)|=0.
\end{equation*}%
\end{itemize}

On the other hand, for $0\leq s\leq t_{2}<t_{1}$:
\begin{itemize}
\item If $\xi _{l-1}\leq s<\xi _{l}$ for some $2\leq l\leq m-1$, then 
\begin{equation*}
\begin{split}
|G(t_{1},s)-G(t_{2},s)|\leq & \,\frac{1}{\gamma }\,|h_{l}(s)|\,e^{-\frac{k\,s%
}{2}}\,\left\vert -e^{-\frac{k\,t_{1}}{2}}\,\sin (\gamma \,t_{1})+e^{-\frac{%
k\,t_{2}}{2}}\,\sin (\gamma \,t_{2})\right\vert \\
& +\frac{1}{\gamma }\,e^{-\frac{k\,s}{2}}\left\vert e^{-\frac{k\,t_{1}}{2}%
}\,\sin (\gamma \,(s-t_{1}))-e^{-\frac{k\,t_{2}}{2}}\,\sin (\gamma
\,(s-t_{2}))\right\vert .
\end{split}%
\end{equation*}%
Last term in the previous sum can be upperly bounded as follows 
\begin{equation*}
\begin{split}
& \left\vert e^{-\frac{k\,t_{1}}{2}}\,\sin (\gamma \,(s-t_{1}))-e^{-\frac{%
k\,t_{2}}{2}}\,\sin (\gamma \,(s-t_{2}))\right\vert \\
=& |\sin (\gamma \,s)|\,\left\vert e^{-\frac{k\,t_{1}}{2}}\,\cos (\gamma
\,t_{1})-e^{-\frac{k\,t_{2}}{2}}\,\cos (\gamma \,t_{2})\right\vert \\
& +|\cos (\gamma \,s)|\,\left\vert -e^{-\frac{k\,t_{1}}{2}}\,\sin (\gamma
\,t_{1})+e^{-\frac{k\,t_{2}}{2}}\,\sin (\gamma \,t_{2})\right\vert \\
\leq & \,\left\vert e^{-\frac{k\,t_{1}}{2}}\,\cos (\gamma \,t_{1})-e^{-\frac{%
k\,t_{2}}{2}}\,\cos (\gamma \,t_{2})\right\vert \\
& +\,\left\vert -e^{-\frac{k\,t_{1}}{2}}\,\sin (\gamma \,t_{1})+e^{-\frac{%
k\,t_{2}}{2}}\,\sin (\gamma \,t_{2})\right\vert .
\end{split}%
\end{equation*}%
As a consequence, 
\begin{equation*}
\begin{split}
|G(t_{1},s)-G(t_{2},s)|\leq & \,\frac{1}{\gamma }\,\left(
|h_{l}(s)|+1\right) \,e^{-\frac{k\,s}{2}}\,\left\vert -e^{-\frac{k\,t_{1}}{2}%
}\,\sin (\gamma \,t_{1})+e^{-\frac{k\,t_{2}}{2}}\,\sin (\gamma
\,t_{2})\right\vert \\
& +\frac{1}{\gamma }\,e^{-\frac{k\,s}{2}}\left\vert e^{-\frac{k\,t_{1}}{2}%
}\,\cos (\gamma \,t_{1})-e^{-\frac{k\,t_{2}}{2}}\,\cos (\gamma
\,t_{2})\right\vert .
\end{split}%
\end{equation*}%
\item If $s\geq \xi _{m-1}$, then 
\begin{equation*}
\begin{split}
|G(t_{1},s)-G(t_{2},s)|=& \frac{1}{\gamma }\,e^{-\frac{k\,s}{2}}\left\vert
e^{-\frac{k\,t_{1}}{2}}\,\sin (\gamma \,(s-t_{1}))-e^{-\frac{k\,t_{2}}{2}%
}\,\sin (\gamma \,(s-t_{2}))\right\vert \\
\leq & \,\frac{1}{\gamma }\,e^{-\frac{k\,s}{2}}\left\vert e^{-\frac{k\,t_{1}%
}{2}}\,\cos (\gamma \,t_{1})-e^{-\frac{k\,t_{2}}{2}}\,\cos (\gamma
\,t_{2})\right\vert \\
& +\,\frac{1}{\gamma }\,e^{-\frac{k\,s}{2}}\left\vert -e^{-\frac{k\,t_{1}}{2}%
}\,\sin (\gamma \,t_{1})+e^{-\frac{k\,t_{2}}{2}}\,\sin (\gamma
\,t_{2})\right\vert .
\end{split}%
\end{equation*}%
Therefore, we can affirm that for a given $\varepsilon >0$ there exists some 
$\delta >0$ such that if $|t_{1}-t_{2}|<\delta $ then, for $s\in \lbrack
0,t_{2})\cup (t_{1},\infty )$, it holds that 
\begin{equation*}
|G(t_{1},s)-G(t_{2},s)|\leq \varepsilon \,e^{-\frac{k\,s}{2}}.
\end{equation*}%
This implies that the first and third terms of the last part of inequality %
\eqref{e:equiconv_1} tend to zero with independence of the function $u\in B$.
\end{itemize}

On the other hand, for $0\leq t_{2}\leq s\leq t_{1}$:
\begin{itemize}
\item If $\xi _{l-1}\leq s<\xi _{l}$ for some $2\leq l\leq m-1$, then 
\begin{equation*}
\begin{split}
|G(t_{1},s)-G(t_{2},s)|\leq & \,\frac{1}{\gamma }\,e^{-\frac{ks}{2}%
}\,|h_{l}(s)|\left\vert -e^{-\frac{kt_{1}}{2}}\,\sin (\gamma t_{1})+e^{-%
\frac{kt_{2}}{2}}\,\sin (\gamma t_{2})\right\vert \\
& +\frac{1}{\gamma }\,e^{-\frac{ks}{2}}\,\left\vert e^{-\frac{kt_{1}}{2}%
}\,\sin (\gamma (s-t_{1}))\right\vert .
\end{split}%
\end{equation*}

\item If $s\geq \xi _{m-1}$, then%
\begin{equation*}
|G(t_{1},s)-G(t_{2},s)|=\frac{1}{\gamma }\,e^{-\frac{ks}{2}}\,\left\vert e^{-%
\frac{kt_{1}}{2}}\,\sin (\gamma (s-t_{1}))\right\vert .
\end{equation*}%
\end{itemize}
Thus, when $s\in \lbrack t_{2},t_{1}]$, it holds that 
\begin{equation*}
|G(t_{1},s)-G(t_{2},s)|\leq C\,e^{-\frac{ks}{2}},
\end{equation*}%
for some constant $C$. This implies that $|G(t_{1},\cdot )-G(t_{2},\cdot
)|\left( \varphi _{r}(\cdot )+(k+M)\,r\right) \in L^{1}[t_{1},t_{2}]$ for
any $t_{1},\,t_{2}\in \lbrack 0,\infty )$. Then it is clear that 
\begin{equation*}
\int_{t_{2}}^{t_{1}}\left\vert G(t_{1},s)-G(t_{2},s)\right\vert \,\left(
\varphi _{r}(s)+(k+M)\,r\right) \mbox{\,$\rm d$}s\xrightarrow[t_1\to t_2]{}0
\end{equation*}%
with independence of the function $u\in B$.

Thus we conclude that given $\varepsilon >0$ there exists $\delta
(\varepsilon )>0$ such that if $|t_{1}-t_{2}|<\delta $, then $%
|Tu(t_{1})-Tu(t_{2})|<\varepsilon $ for all $u\in B$.

In a completely analogous way, finding suitable upper bounds for $\left\vert 
\frac{\partial \,G}{\partial \,t}(t_{1},s)-\frac{\partial \,G}{\partial \,t}%
(t_{2},s)\right\vert $, it is possible to prove that given $\varepsilon >0$
there exists $\delta (\varepsilon )>0$ such that if $|t_{1}-t_{2}|<\delta $,
then $|(Tu)^{\prime }(t_{1})-(Tu)^{\prime }(t_{2})|<\varepsilon $ for all $%
u\in B$. 

Therefore, $T(B)$ is equicontinuous.

\item[(iii)] $T(B)$ is equiconvergent at $\infty $:\newline
Given $u\in B$, it holds that 
\begin{equation*}\begin{split}
\left\vert Tu(t)-\lim\limits_{t\rightarrow \infty }Tu(t)\right\vert
\le & \,  C_{1}\,e^{-\frac{kt}{2}}\,\left( \int_{0}^{\infty }e^{-\frac{ks}{2}}\,\varphi _{r}(s)\,\dif s+\left( 2+\frac{2\,M}{k}\right)
\,r\right) \\[2pt]
&+ \lim\limits_{t\rightarrow \infty } C_{1}\,e^{-\frac{kt}{2}}\,\left( \int_{0}^{\infty }e^{-\frac{ks}{2}}\,\varphi _{r}(s)\,\dif s+\left( 2+\frac{2\,M}{k}\right)
\,r\right) \\[2pt]
\leq &\, C_{1}\,e^{-\frac{kt}{2}}\,\left( \int_{0}^{\infty }e^{-\frac{ks%
	}{2}}\,\varphi _{r}(s)\,\dif s+\left( 2+\frac{2\,M}{k}\right)
\,r\right) \xrightarrow[t\to \infty]{}0
\end{split}\end{equation*}
and 
\begin{equation*}\begin{split}
\left\vert (Tu)^{\prime }(t)-\lim\limits_{t\rightarrow \infty }(Tu)^{\prime
}(t)\right\vert \leq & \, C_{2}\,e^{-\frac{kt}{2}}\,\left(
\int_{0}^{\infty }e^{-\frac{ks}{2}}\,\varphi _{r}(s)\,\mbox{\,$\rm d$}%
s+\left( 2+\frac{2\,M}{k}\right) \,r\right) \\[2pt]
& + \lim\limits_{t\rightarrow \infty } C_{2}\,e^{-\frac{kt}{2}}\,\left(
\int_{0}^{\infty }e^{-\frac{ks}{2}}\,\varphi _{r}(s)\,\mbox{\,$\rm d$}%
s+\left( 2+\frac{2\,M}{k}\right) \,r\right) \\[2pt]
\leq & \, C_{2}\,e^{-\frac{kt}{2}}\,\left(
\int_{0}^{\infty }e^{-\frac{ks}{2}}\,\varphi _{r}(s)\,\mbox{\,$\rm d$}%
s+\left( 2+\frac{2\,M}{k}\right) \,r\right) \xrightarrow[t\to \infty]{}0,
\end{split}\end{equation*}
that is, $TB$ is equiconvergent at $\infty $.\newline
Therefore, from Theorem \ref{thm_comp_criteria_corduneanu}, we conclude that 
$T(B)$ is relatively compact in $X$.
\end{description}
\end{proof}

Now we will see our existence results.

\begin{theorem}
\label{t-exist1-multipoint} Let $f:[0,\infty )\times {\mathbb{R}}%
^{2}\rightarrow {\mathbb{R}}$ be such that there exists $t_{0}\in \lbrack
0,\infty )$ for which $f(t_{0},0,0)\neq 0$. Moreover, suppose that, for $%
C_{1}$ and $C_{2}$ given in Remark \ref{r-bound-G-Gprima}, either

\begin{itemize}
\item $(H_{1})$ holds and, moreover, there exists some $R>0$ such that 
\begin{equation}
\begin{split}
\max \{C_{1},\,C_{2}\}\max \left\{ \sup_{t>\xi _{m-1}}e^{-\frac{k\,t}{2}%
}\int_{0}^{t}e^{-\frac{k\,s}{2}}\,\varphi _{R}(s)\,\dif s,\,\int_{0}^{\xi
_{m-1}}e^{-\frac{k\,s}{2}}\,\varphi _{R}(s)\,\dif s\right\} &  \\
+\max \{C_{1},\,C_{2}\}\,\max \left\{ \frac{1}{2},\,2\left( 1-e^{-\frac{%
k\,\xi _{m-1}}{2}}\right) \right\} \left( 1+\frac{M}{k}\right) \,R& <R,
\end{split}
\label{H1}
\end{equation}
\end{itemize}
or
\begin{itemize}
\item $(H_{2})$ holds and, moreover, there is $R>0$ such that 
\begin{equation*}
\begin{split}
\max \{C_{1},\,C_{2}\}\max \left\{ \sup_{t>\xi _{m-1}}e^{-\frac{k\,t}{2}%
}\int_{0}^{t}e^{-\frac{k\,s}{2}}\,\phi _{R}(s)\,\dif s,\,\int_{0}^{\xi
_{m-1}}e^{-\frac{k\,s}{2}}\,\phi _{R}(s)\,\dif s\right\} &  \\
+\max \{C_{1},\,C_{2}\}\,\max \left\{ \frac{1}{2},\,2\left( 1-e^{-\frac{%
k\,\xi _{m-1}}{2}}\right) \right\} \left( 1+\frac{M}{k}\right) \,R& <R.
\end{split}%
\end{equation*}
\end{itemize}

Then problem \eqref{e-multipoint} has at least a nontrivial solution.
\end{theorem}

\begin{proof}
We will prove the first case, being the second one analogous.

Consider 
\begin{equation*}
D=\{u\in X:\ \Vert u\Vert <R\}.
\end{equation*}%
If $u\in D$ then, 
\begin{equation*}
|Tu(t)|\leq \int_{0}^{\infty }\left\vert G(t,s)\right\vert \,\left( \varphi
_{R}(s)+\left( k+M\right) R\right) \,\dif s,\ \forall \,t\in \lbrack
0,\infty ),
\end{equation*}%
and, since $G(t,s)=0$ for $s\geq \max \{t,\,\xi _{m-1}\}$, 
\begin{equation*}
|Tu(t)|\leq \int_{0}^{\max \{t,\,\xi_{m-1}\}}\left\vert G(t,s)\right\vert
\,\left( \varphi _{R}(s)+\left( k+M\right) R\right) \,\dif s,\ \forall
\,t\in \lbrack 0,\infty ).
\end{equation*}%
If $t>\xi _{m-1}$, the previous expression leads to 
\begin{equation*}
\begin{split}
|Tu(t)|& \leq \int_{0}^{t}\left\vert G(t,s)\right\vert \,\left( \varphi
_{R}(s)+\left( k+M\right) R\right) \,\dif s\\ &\leq C_{1}\,e^{-\frac{k\,t}{2}%
}\int_{0}^{t}e^{-\frac{k\,s}{2}}\,\left( \varphi _{R}(s)+\left( k+M\right)
R\right) \,\dif s \\
& \leq C_{1}\,\left( e^{-\frac{k\,t}{2}}\int_{0}^{t}e^{-\frac{k\,s}{2}%
}\,\varphi _{R}(s)\,\dif s+2\,e^{-\frac{k\,t}{2}}\left( 1-e^{-\frac{k\,t}{2}%
}\right) \left( 1+\frac{M}{k}\right) R\right)  \\
& \leq C_{1}\,\left( e^{-\frac{k\,t}{2}}\int_{0}^{t}e^{-\frac{k\,s}{2}%
}\,\varphi _{R}(s)\,\dif s+\frac{1}{2}\left( 1+\frac{M}{k}\right) R\right) .
\end{split}%
\end{equation*}%
On the other hand, if $t\leq \xi _{m-1}$, we obtain that 
\begin{equation*}
\begin{split}
|Tu(t)|& \leq \int_{0}^{\xi _{m-1}}\left\vert G(t,s)\right\vert \,\left(
\varphi _{R}(s)+\left( k+M\right) R\right) \,\dif s \\
& \leq C_{1}\,e^{-\frac{k\,t%
}{2}}\int_{0}^{\xi _{m-1}}e^{-\frac{k\,s}{2}}\,\left( \varphi _{R}(s)+\left(
k+M\right) R\right) \,\dif s \\
& \leq C_{1}\,\left( e^{-\frac{k\,t}{2}}\int_{0}^{\xi _{m-1}}e^{-\frac{k\,s}{%
2}}\,\varphi _{R}(s)\,\dif s+2\,e^{-\frac{k\,t}{2}}\left( 1-e^{-\frac{k\,\xi
_{m-1}}{2}}\right) \left( 1+\frac{M}{k}\right) R\right)  \\
& \leq C_{1}\,\left( \int_{0}^{\xi_{m-1}}e^{-\frac{k\,s}{2}}\,\varphi
_{R}(s)\,\dif s+2\left( 1-e^{-\frac{k\,\xi_{m-1}}{2}}\right) \left( 1+\frac{M}{k}\right) R\right) .
\end{split}%
\end{equation*}

Therefore, 
\begin{equation*}
\begin{split}
|Tu(t)|\leq & \,C_{1}\max \left\{ \sup_{t>\xi _{m-1}}e^{-\frac{k\,t}{2}%
}\int_{0}^{t}e^{-\frac{k\,s}{2}}\,\varphi _{R}(s)\,\dif s,\,\int_{0}^{\xi
_{m-1}}e^{-\frac{k\,s}{2}}\,\varphi _{R}(s)\,\dif s\right\} \\
& +C_{1}\,\max \left\{ \frac{1}{2},\,2\left( 1-e^{-\frac{k\,\xi _{m-1}}{2}%
}\right) \right\} \left( 1+\frac{M}{k}\right) \,R,\quad \forall \,t\in
\lbrack 0,\infty ).
\end{split}%
\end{equation*}

Analogously, it can be seen that 
\begin{equation*}
\begin{split}
|(Tu)^{\prime }(t)|\leq & \,C_{2}\max \left\{ \sup_{t>\xi _{m-1}}e^{-\frac{%
k\,t}{2}}\int_{0}^{t}e^{-\frac{k\,s}{2}}\,\varphi _{R}(s)\,\dif %
s,\,\int_{0}^{\xi _{m-1}}e^{-\frac{k\,s}{2}}\,\varphi _{R}(s)\,\dif s\right\}
\\
& +C_{2}\,\max \left\{ \frac{1}{2},\,2\left( 1-e^{-\frac{k\,\xi _{m-1}}{2}%
}\right) \right\} \left( 1+\frac{M}{k}\right) \,R,\quad \forall \,t\in
\lbrack 0,\infty ).
\end{split}%
\end{equation*}

Thus, by (\ref{H1}), 
\begin{equation*}
\begin{split}
\Vert Tu\Vert \leq & \,\max \{C_{1},\,C_{2}\}\max \left\{ \sup_{t>\xi
_{m-1}}e^{-\frac{k\,t}{2}}\int_{0}^{t}e^{-\frac{k\,s}{2}}\,\varphi _{R}(s)\,%
\dif s,\,\int_{0}^{\xi _{m-1}}e^{-\frac{k\,s}{2}}\,\varphi _{R}(s)\,\dif 
s\right\} \\
& +\max \{C_{1},\,C_{2}\}\,\max \left\{ \frac{1}{2},\,2\left( 1-e^{-\frac{
k\,\xi _{m-1}}{2}}\right) \right\} \left( 1+\frac{M}{k}\right) \,R<R,
\end{split}%
\end{equation*}%
that is, $Tu\in D$.

Therefore, $TD\subset D$ and, from Theorem \ref{t-Schauder}, the operator $T$
has at least one fixed point in $D$, which is a solution of problem %
\eqref{e-multipoint}. Moreover, since there exists at least some value $%
t_{0}\in \lbrack 0,\infty )$ for which $f(t_{0},0,0)\neq 0$, this solution
can not be the trivial one.
\end{proof}

Now, we will give another existence result based on the lower and upper
solutions technique:

\begin{theorem}
\label{t-exist2-multipoint} Let $\alpha ,\,\beta \in X$ be lower and upper
solutions of problem \eqref{e-multipoint}, respectively, with 
\begin{equation*}
\alpha (t)\leq \beta (t),\quad \forall \,t\in \lbrack 0,\infty ),
\end{equation*}%
and denote 
\begin{equation}
\widetilde{R}=\max \{\Vert \alpha \Vert _{\infty },\,\Vert \beta \Vert
_{\infty }\}.  \label{Rtil}
\end{equation}

Assume that the nonlinearity $f(t,x,y)$ is nondecreasing in $y$ and, suppose
that, for $C_{1}$ and $C_{2}$ given by Remark \ref{r-bound-G-Gprima}, either

\begin{itemize}
\item $(H_{1})$ holds and, moreover, there exists some $R>0$ such that%
\begin{equation*}
\begin{split}
\max \{C_{1},\,C_{2}\}\max \left\{ \sup_{t>\xi _{m-1}}e^{-\frac{k\,t}{2}%
}\int_{0}^{t}e^{-\frac{k\,s}{2}}\varphi \,_{\max \{R,\widetilde{R}\}}(s)\,%
\dif s,\,\int_{0}^{\xi _{m-1}}e^{-\frac{k\,s}{2}}\,\varphi _{\max \{R,%
\widetilde{R}\}}(s)\,\dif s\right\} &  \\
+\max \{C_{1},\,C_{2}\}\,\max \left\{ \frac{1}{2},\,2\left( 1-e^{-\frac{%
k\,\xi _{m-1}}{2}}\right) \right\} \left( 1+\frac{M}{k}\right) R& <R.
\end{split}%
\end{equation*}
\end{itemize}
or
\begin{itemize}
\item $(H_{2})$ holds and, moreover, there exists some $R>0$ such that 
\begin{equation*}
\begin{split}
\max \{C_{1},\,C_{2}\}\max \left\{ \sup_{t>\xi _{m-1}}e^{-\frac{k\,t}{2}%
}\int_{0}^{t}e^{-\frac{k\,s}{2}}\,\phi _{\max \{R,\widetilde{R}\}}(s)\,\dif %
s,\,\int_{0}^{\xi _{m-1}}e^{-\frac{k\,s}{2}}\,\phi _{\max \{R,\widetilde{R}%
\}}(s)\,\dif s\right\} &  \\
+\max \{C_{1},\,C_{2}\}\,\max \left\{ \frac{1}{2},\,2\left( 1-e^{-\frac{%
k\,\xi _{m-1}}{2}}\right) \right\} \left( 1+\frac{M}{k}\right) R& <R.
\end{split}%
\end{equation*}
\end{itemize}

Then, problem \eqref{e-multipoint} has a solution $u\in X$ such that 
\begin{equation*}
\alpha(t)\le u(t) \le \beta(t), \quad \forall\,t\in [0,\infty).
\end{equation*}
\end{theorem}

\begin{proof}
We will prove the first case, being the second one analogous.

Let $\varepsilon>0$ be such that 
\begin{equation*}
\begin{split}
\max\{C_1,\,C_2\} \max\left\{ \sup_{t>\xi_{m-1}} e^{-\frac{k\,t}{2}}
\int_{0}^{t}e^{-\frac{k\,s}{2}}\, \varphi_{\max\{R,\widetilde{R}\}}(s)\, %
\dif s, \, \int_{0}^{\xi_{m-1}}e^{-\frac{k\,s}{2}}\, \varphi_{\max\{R,%
\widetilde{R}\}}(s)\, \dif s \right\} & \\
+ \max\{C_1,\,C_2\} \, \max\left\{ \frac{1}{2}, \, 2\left(1-e^{-\frac{%
k\,\xi_{m-1}}{2}}\right) \right\}\left(\left(1+\frac{M}{k}\right) R+\frac{%
\varepsilon}{k}\left(R+\widetilde{R}\right)\right) & < R.
\end{split}%
\end{equation*}

Consider the modified problem 
\begin{equation}  \label{e-multipoint-mod-lowup}
\left\{%
\begin{split}
&u^{\prime \prime }(t)+k\,u^{\prime }(t)+M\,u(t) =
f(t,\delta(t,u(t)),u^{\prime }(t))+k\,u^{\prime
}(t)+M\,u(t)+\varepsilon\left(u(t)-\delta(t,u(t))\right), \quad t\in
[0,\infty), \\
&u(0)=0, \ u^{\prime }(+\infty)=\sum_{i=1}^{m-1}\alpha_i\,u^{\prime }(\xi_i),
\end{split}%
\right.
\end{equation}
where the function $\delta\colon [0,\infty)\times {\mathbb{R}}\rightarrow {%
\mathbb{R}}$ is given by 
\begin{equation*}
\delta(t,u(t))=\left\{%
\begin{array}{ll}
\beta(t), & u(t)>\beta(t), \\[2pt] 
u(t), & \alpha(t)\le u(t) \le \beta(t), \\[2pt] 
\alpha(t), & u(t)<\alpha(t).%
\end{array}
\right.
\end{equation*}

Define now operator $T^*\colon X \rightarrow X$ by 
\begin{equation*}
T^*u(t)=\int_{0}^{\infty}G(t,s) \left(f(s,\delta(s,u(s)),u^{\prime }(s))
+k\,u^{\prime }(s) +M\,u(s)+\varepsilon(u(s)-\delta(s,u(s))) \right).
\end{equation*}

Following the same steps as in Lemma \ref{l-T-comp-cont-multipoint}, it is
easy to prove that if $(H_{1})$ holds, then $T^{\ast }$ is well-defined in $%
X $ and it is a completely continuous operator.

Moreover, it is clear that, by (\ref{Rtil}), $|\delta (t,u(t))|\leq 
\widetilde{R}$ for all $t\in \lbrack 0,\infty )$. Thus, if we consider 
\begin{equation*}
D=\{u\in X:\ \Vert u\Vert <R\}
\end{equation*}%
and $u\in D$ then, following analogous steps than in the proof of Theorem %
\ref{t-exist1-multipoint}, it can be deduced that 
\begin{equation*}
\begin{split}
\Vert T^{\ast }u\Vert \leq & \,\max \{C_{1},\,C_{2}\}\max \left\{
\sup_{t>\xi _{m-1}}e^{-\frac{k\,t}{2}}\int_{0}^{t}e^{-\frac{k\,s}{2}%
}\,\varphi _{\max \{R,\widetilde{R}\}}(s)\,\dif s,\,\int_{0}^{\xi _{m-1}}e^{-%
\frac{k\,s}{2}}\,\varphi _{\max \{R,\widetilde{R}\}}(s)\,\dif s\right\} \\
& +\max \{C_{1},\,C_{2}\}\,\max \left\{ \frac{1}{2},\,2\left( 1-e^{-\frac{%
k\,\xi _{m-1}}{2}}\right) \right\} \left( \left( 1+\frac{M}{k}\right) R+%
\frac{\varepsilon }{k}\left( R+\widetilde{R}\right) \right) <R,
\end{split}%
\end{equation*}%
that is, $T^{\ast }u\in D$.

Therefore, $TD\subset D$ and, from Theorem \ref{t-Schauder}, $T^{\ast }$ has
at least one fixed point in D, which corresponds to a solution of problem %
\eqref{e-multipoint-mod-lowup}.

Finally, we will prove that this solution $u$ of problem %
\eqref{e-multipoint-mod-lowup} satisfies that 
\begin{equation*}
\alpha(t)\le u(t) \le \beta(t), \quad \forall\,t\in[0,\infty),
\end{equation*}
which implies that it is also a solution of problem \eqref{e-int-T}.

Define $v(t)=u(t)-\beta (t)$ and consider 
\begin{equation*}
v(t_{0}):=\sup \{v(t):\ t\in \lbrack 0,\infty ]\}.
\end{equation*}%
Suppose that $v(t_{0})>0$. Then, since 
\begin{equation*}
v(0)=-\beta (0)\leq 0,
\end{equation*}%
necessarily $t_{0}\neq 0$. Thus, we have two possibilities:

\begin{itemize}
\item[(1)] If $t_{0}\in (0,\infty ),$ then 
\begin{equation*}
v(t_{0})=\max \{v(t):\ t\in (0,\infty )\}=\max \{v(t):\ t\in \lbrack
0,\infty )\}.
\end{equation*}%
In this case, $v^{\prime }(t_{0})=0$ and, moreover, there exists $t_{1}\in
(0,\infty )$ such that $v(t)>0$, $v^{\prime }(t)>0$ and $v^{\prime \prime
}(t)\leq 0$ a.\thinspace e. on $(t_{1},t_{0})$.

\item[(2)] If $t_{0}=+\infty $, then there exists $t_{2}\in (0,\infty )$
such that $v(t)>0$ and $v^{\prime }(t)>0$ on $(t_{2},\infty )$. Moreover, as 
$v\in X$, it is upperly bounded, and, since $v^{\prime }(t)>0$ on $%
(t_{2},\infty )$, it must occur that $v^{\prime }(+\infty )=0$. Thus, there
exists $t_{3}\geq t_{2}$ such that $v^{\prime \prime }(t)\leq 0$
a.\thinspace e. on $(t_{3},\infty )$.
\end{itemize}

Therefore, in both cases, there exists $t_{\ast }\in [0,\infty )$ such that $%
v(t)>0$, $v^{\prime }(t)>0$ and $v^{\prime \prime }(t)\leq 0$ a.\thinspace
e. on $(t_{\ast },t_{0})$. Then, 
\begin{equation*}
\int_{t_{\ast }}^{t_{0}}v^{\prime \prime }(t)\,\dif t\leq 0.
\end{equation*}

On the other hand, we reach the contradiction 
\begin{equation*}
\begin{split}
0& \geq \int_{t_{\ast }}^{t_{0}}v^{\prime \prime }(t)\,\dif t=\int_{t_{\ast
}}^{t_{0}}\left( u^{\prime \prime }(t)-\beta ^{\prime \prime }(t)\right) \,%
\dif t \\
& \geq \int_{t_{\ast }}^{t_{0}}\left( f(t,\delta (t,u(t)),u^{\prime
}(t))+\varepsilon \left( u(t)-\delta (t,u(t))\right) -f(t,\beta (t),\beta
^{\prime }(t))\right) \,\dif t \\
& =\int_{t_{\ast }}^{t_{0}}\left( f(t,\beta (t),u^{\prime }(t))-f(t,\beta
(t),\beta ^{\prime }(t))+\varepsilon \left( u(t)-\beta (t)\right) \right) \,%
\dif t \\
& \geq \varepsilon \int_{t_{\ast }}^{t_{0}}\left( u(t)-\beta (t)\right) \,%
\dif t=\varepsilon \int_{t_{\ast }}^{t_{0}}v(t)\,\dif t>0.
\end{split}%
\end{equation*}%
Therefore 
\begin{equation*}
\sup \{v(t):\ t\in \lbrack 0,\infty )\}\leq 0,
\end{equation*}%
that is, 
\begin{equation*}
u(t)\leq \beta (t),\quad t\in \lbrack 0,\infty ).
\end{equation*}

Analogously, it can be seen that 
\begin{equation*}
u(t)\ge \alpha(t), \quad t\in[0,\infty).
\end{equation*}

Therefore, $u$ is a solution of problem \eqref{e-multipoint}.
\end{proof}

\section{Example}

Let us consider the following boundary value problem: 
\begin{equation}  \label{ex1-multipoint}
\left\{%
\begin{split}
& u''(t)=\frac{1}{1000}\,(2+\sin t)\,e^{-|u(t)|} \, \frac{|1-u(t)|}{(u(t))^2+1}\,\left(u'(t)-1\right), \quad t\in[0,\infty), \\[2pt]
& u(0)=0, \ u^{\prime }(+\infty)=0.11 \,u^{\prime }(0)+0.89 \,u^{\prime
}(0.11).
\end{split}%
\right.
\end{equation}
This problem is a particular case of \eqref{e-multipoint} with $%
f(t,x,y)=\frac{1}{1000}\,(2+\sin t)\,e^{-|x|}\,\frac{|1-x|}{x^2+1}\,(y-1)$, $m=3$, $\alpha_1=0.11$, $%
\alpha_2=0.89$, $\xi_1=0$ and $\xi_2=0.11$.

We have that for $|x|,\,|y|<r$, it holds that 
\begin{equation*}
|f(t,x,y)|\le \frac{1}{1000}\, (2+\sin t)\,(r+1)^2,
\end{equation*}
so we could take $\phi_r(t)=\frac{1}{1000}\, (2+\sin t)\,(r+1)^2$ and hypothesis $(H_2)$
holds. We note that, since $\phi_r\notin L^1[0,\infty)$, results in \cite{JiangYang2016} can not be applied to solve this problem.

We will look for a pair of lower and upper solutions of problem %
\eqref{e-multipoint} and suitable values for $k$ and $M$ for which the
hypotheses in Theorem \ref{t-exist2-multipoint} hold.

As lower and upper solutions we will take 
\begin{equation*}
\alpha(t)=\frac{3}{400}\left(-(t+1)\,e^{-t}+\frac{t^2-t}{t^2+1}\right) \quad \text{and} \quad
\beta(t)=1, \quad \forall\,t\in[0,\infty).
\end{equation*}

It can be checked that $\|\alpha\|_\infty\approx 0.0087$ and $\|\beta\|_\infty=1$. Therefore, we obtain that $\widetilde{R}$ given in \eqref{Rtil} is
\begin{equation*}
\widetilde{R}=1.
\end{equation*}
Moreover, for $M=0.35$ and $k=0.86$, we obtain the following approximations for $C_1$ and $C_2$:
\[C_1\approx 1.2305, \quad C_2\approx 1.3395.\]
Therefore,
\begin{equation*}
\max\{C_1,C_2\}\,\max\left\{\frac{1}{2},2\left( 1-e^{-\frac{k\,\xi_{m-1}}{2}%
}\right) \right\} \left(1+\frac{M}{k}\right)\approx 0.9423.
\end{equation*}

On the other hand, 
\[ \int_{0}^{\xi_{m-1}}e^{-\frac{k\,s}{2}}\, \phi_{\max\{R,\widetilde{R}%
	\}}(s)\, \dif s \approx 0.00022 \,\left( \max\{R,\widetilde{R}\} +1\right)^2\]
and
\[\sup_{t>\xi_{m-1}} e^{-\frac{k\,t}{2}} \int_{0}^{t} e^{-\frac{k\,s}{2}}\, \phi_{\max\{R,\widetilde{R}\}}(s)\, \dif s \approx 0.00174\, \left( \max\{R,\widetilde{R}\} +1\right)^2. \]

Therefore, we can approximate 
\begin{equation*}
\begin{split}
\max\{C_1,\,C_2\} \max\left\{ \sup_{t>\xi_{m-1}} e^{-\frac{k\,t}{2}}
\int_{0}^{t}e^{-\frac{k\,s}{2}}\, \phi_{\max\{R,\widetilde{R}\}}(s)\, \dif %
s, \, \int_{0}^{\xi_{m-1}}e^{-\frac{k\,s}{2}}\, \phi_{\max\{R,\widetilde{R}%
\}}(s)\, \dif s \right\} & \\
+ \max\{C_1,\,C_2\} \, \max\left\{ \frac{1}{2}, \, 2\left(1-e^{-\frac{%
k\,\xi_{m-1}}{2}}\right) \right\}\left(1+\frac{M}{k}\right) R & \\
\approx 0.00233 \,\left( \max\{R,1\} +1\right)^2+0.9423\,R,
\end{split}%
\end{equation*}
and it can be seen that for $R\in (R_0,R_1)$, with $R_0\approx 0.1615$ and $R_1\approx 22.7199$, it holds that
\begin{equation*}
0.00233 \,\left( \max\{R,1\} +1\right)^2 +0.9423\,R< R.
\end{equation*}

Therefore, we have proved the existence of a solution $u$ of problem \eqref{ex1-multipoint} such that 
\begin{equation*}
\frac{3}{400}\left(-(t+1)\,e^{-t}+\frac{t^2-t}{t^2+1}\right)\leq u(t)\leq 1,\quad \forall \,t\in
\lbrack 0,\infty ),
\end{equation*}%
which implies that this solution is non trivial.

\section*{Declarations}
\subsection*{Availability of data and material}
Not applicable.

\subsection*{Competing interests}
The authors declare that there is no conflict of interest regarding the publication of this article.

\subsection*{Funding}
First author was partially supported by Xunta de Galicia (Spain), project
EM2014/032, AIE Spain and FEDER, grants MTM2013-43014-P, MTM2016-75140-P,
and FPU scholarship, Ministerio de Educaci\'{o}n, Cultura y Deporte, Spain.

\subsection*{Authors' contributions}
Each author equally contributed to this paper, and read and approved the final manuscript.

\subsection*{Acknowledgments}
Not applicable.

\end{document}